\documentclass[12pt]{article}

\usepackage{amsfonts}
\usepackage{tikz}
\usepackage{pgfplots}

\newtheorem{theorem}{Theorem}

\newtheorem{lemma}[theorem]{Lemma}

\newtheorem{remark}[theorem]{Remark}

\newenvironment{proof}[1][Proof]{\noindent\textbf{#1.} }{\ \rule{0.5em}{0.5em}}
\begin{document}

\title{Convex subsets of non-convex Lorentz balls}
\author{Daniel J. Fresen\thanks{%
University of Pretoria, Department of Mathematics and Applied Mathematics,
daniel.fresen@up.ac.za MSC: 52A27, 52A30, 60E05, 60E15.}}
%\date{}
\maketitle

\begin{abstract}
Many star bodies have convex subsets with approximately the same Gaussian measure (of the complement). Inspired by this phenomenon, and in connection with the randomized Dvoretzky theorem for Lorentz spaces, we derive bounds on the distribution of certain functions of a Gaussian random vector by approximating their sub-level sets by convex subsets.
\end{abstract}

\section{Introduction}

\noindent \underline{Our starting point}

\bigskip

Let $E\subset\mathbb{R}^n$ be a star convex Borel set, i.e. for all $x\in E$ and $\lambda\in [0,1]$, $\lambda x\in E$. For many such $E$ there exists a convex set $K\subseteq E$ such that
\begin{equation}
\gamma_n\left(\mathbb{R}^n\setminus K\right)\leq C\gamma_n\left(\mathbb{R}^n\setminus C^{-1}E\right)\label{conv app}
\end{equation}
where $\gamma_n$ is the standard Gaussian measure on $\mathbb{R}^n$ and $C>1$ is a universal constant. Several elementary observations about this type of approximation are contained in \cite{Fr}, and we refer the reader to Talagrand's papers \cite{Tal1, Tal2} for a more in-depth discussion of the general topic of approximating star convex sets by convex sets, including various conjectures.

Eq. (\ref{conv app}) expresses a sense in which the bound
\begin{equation}
\gamma_n\left(\mathbb{R}^n\setminus E\right)\leq\gamma_n\left(\mathbb{R}^n\setminus K\right)\label{setincbound}
\end{equation}
is sharp and gives hope that the Gaussian measure of a star body can be adequately estimated by considering convex subsets. An advantage of this is that one can use the Gaussian concentration inequality to bound the Gaussian measure of a convex body, since its Minkowski functional is Lipschitz and there is a convenient expression for its Lipschitz constant (as its supremum over the sphere).

%\bigskip
\newpage

\noindent \underline{Examples of (\ref{conv app}) from \cite{Fr}}

\bigskip

(\ref{conv app}) holds if $E$ is of the form $\lambda B_F$ where $F:[0,\infty)\rightarrow[0,\infty)$ is concave and strictly increasing with $F(0)=0$ and $F(1)=1$, $\lambda>1$, and
\[
B_F=\left\{x\in\mathbb{R}^n:\sum_{i=1}^nF\left(\left\vert x_i\right\vert \right)\leq n\right\}.
\]
It also holds with $E=\lambda A$ and, say $C=1.01$, provided that $A$ is a bounded star convex set with $0\in\textrm{int}(A)$ whose Minkowski functional is continuous, and $\lambda$ is sufficiently large (depending on $A$). In both cases we may take $K$ to be a Euclidean ball centred at the origin.

\bigskip

\underline{The goal of this paper}

\bigskip

Inspired by these observations we bound $\gamma_n\left(\mathbb{R}^n\setminus E\right)$ by finding a convex subset $K\subseteq E$ and using (\ref{setincbound}). This is implemented in the case where $E$ is a possibly non-convex Lorentz ball.

\bigskip

\underline{Motivation and relationship to \cite{FrL}}

\bigskip

While the topic of approximating star bodies with convex subsets is of independent interest, our interest stems from its connection to the concentration of measure phenomenon as outlined in \cite{Fr} and its application to the randomized Dvoretzky theorem for Lorentz spaces in \cite{FrL}. If $1\leq p <\infty$ and $\omega=(\omega_i)_1^n$ is a non-increasing sequence in $[0,1]$ with $\omega_1=1$, then the corresponding Lorentz norm is given by
\[
\left\vert x \right\vert _{\omega,p}=\left(\sum_{i=1}^{n}\omega_ix_{[i]}^p\right)^{1/p}
\]
where $\left(x_{[i]}\right)_1^n$ is the non-increasing rearrangement of $\left(\left\vert x_i\right\vert\right)_1^n$. If $\varepsilon \in (0,1)$ and $k\leq d(p,\varepsilon,\omega)$, where $d(p,\varepsilon,\omega)$ is a parameter that can be given precisely, and $G$ is an $n\times k$ random matrix with i.i.d. $N(0,1)$ entries, then the following event occurs with high probability: for all $x\in\mathbb{R}^k$, 
\[
(1-\varepsilon)M\left\vert x \right\vert \leq \left\vert Gx \right\vert_{\omega,p} \leq (1+\varepsilon)M\left\vert x \right\vert
\]
where $M=\mathbb{E}\left\Vert Ge_1 \right\Vert$. In a series of two paper, this one and \cite{FrL}, we provide lower bounds for $d(p,\varepsilon,\omega)$ that extend results of Paouris, Valettas and Zinn \cite{PVZ} which apply when $\omega$ is a constant sequence, in which case the corresponding Lorentz space is $\ell_p^n$ (our results also improve the dependence on $p$ in this special case). A major difference between our work and \cite{PVZ} is that the Lorentz norm of a Gaussian random vector is typically not written in terms of the sum of i.i.d. random variables.

Of special interest is the case where $\omega_i=i^{-r}$, $0\leq r<\infty$, and the corresponding Lorentz norm is then
\[
\left\vert x \right\vert _{r,p}=\left(\sum_{i=1}^{n}i^{-r}x_{[i]}^p\right)^{1/p}
\]

Setting $X=Gx$ and assuming that $\left\vert x \right\vert =1$ where $\left\vert \cdot \right\vert$ is the Euclidean norm, we see that $X$ has the standard normal distribution in $\mathbb{R}^n$. The critical step in the proof of randomized Dvoretzky theorems, following Milman \cite{Mil} and later work by Gordon \cite{Gord}, Pisier \cite{Pis} and Schechtman \cite{Sch}, is usually to derive concentration estimates for the norm. Unlike the classical approaches, we prove concentration estimates directly for $\left\vert X \right\vert _{r,p}^p=\sum_{i=1}^{n}i^{-r}X_{[i]}^p$ (i.e. to a function of the norm rather than the norm itself). This difference is not insignificant, and we refer the reader to \cite{FrL} for more details of its effect; the main effect being that points where the local Lipschitz constant is large have been moved away from the origin so that the methodology outlined in \cite{Fr} can be applied (this methodology builds upon methods from \cite{ASY, Adam15, BNT, Gra, MeSz, Vu}).

Part of this process involves estimating the distribution of $\left\vert\nabla \psi(X)\right\vert$, where $\psi(x)=\sum_{i=1}^{n}i^{-r}x_{[i]}^p$. This comes down to estimating the distribution of
\begin{equation}
\sum_{i=1}^{n}i^{-2r}X_{[i]}^{2(p-1)}\label{quanttt}
\end{equation}
For $p\geq 3/2$ the exponent of $X_{[i]}$ is at least 1 and the quantity in (\ref{quanttt}) can be written in terms of a norm of $X$, and its distribution estimated using Gaussian concentration. For $1< p<3/2$ we run into challenges not only in terms of estimating the distribution of $\left\vert\nabla \psi(X)\right\vert$, but also in terms of a quasi-norm constant that blows up as $p\rightarrow 1$. This blowup is a technical matter related to the epsilon net argument as used in \cite{FrL}.

\bigskip

\underline{More on this paper}

\bigskip

Our main result, Theorem \ref{orderorderbound}, presents estimates of the form
\[
\mathbb{P}\left\{\left\vert X \right \vert_{\sharp}\leq S\right\}\geq 1-C\exp\left(-t^2/2\right) \hspace{1cm}\textit{and} \hspace{1cm} \left\vert X \right \vert_{\sharp}\leq S \Rightarrow \sum_{i=1}^{n}i^{-2r}X_{[i]}^{2(p-1)}\leq R
\]
where $\left\vert \cdot \right \vert_{\sharp}$ is a norm and $R$, $S$ and $\left\vert \cdot \right \vert_{\sharp}$ depend on $t$. The logical implication on the right means that a sub-level set of $x\mapsto\sum_{i=1}^{n}i^{-2r}x_{[i]}^{2(p-1)}$ contains a (convex) sub-level set of $x\mapsto\left\vert x \right \vert_{\sharp}$. This implies (obviously) that
\[
\mathbb{P}\left\{\sum_{i=1}^{n}i^{-2r}X_{[i]}^{2(p-1)}\leq R\right\}\geq 1-C\exp\left(-t^2/2\right)
\]
but gives more information that allows us, in \cite{FrL}, to avoid the blow-up of the quasi-norm constant mentioned above. Our results are sharp for $t=C$ and as $t\rightarrow\infty$ in a sense made clear in Remark \ref{shar simpy}.

The most interesting case is, unsurprisingly, $1<p<3/2$. The line segment $p=2(1-r)$ (for $1/4<r<1/2$) is of special interest, and it appears that this reflects the underlying geometry (and not only the methods used).

\section{Notation and once-off explanations}

The symbols $C$ and $c$ denote positive universal constants that may take on
different values at each appearance. $\mathbb{M}$ and $\mathbb{E}$ denote median and expected value. $n$ and $k$ will typically denote natural numbers and this will not always be stated explicitly but should be clear from the context. $\textrm{Lip}(f)\in \left[0,\infty\right]$ denotes the Lipschitz constant of any function $f:\mathbb{R}^n \rightarrow \mathbb{R}$ with respect to the Euclidean norm. $\sum_{i=a}^bf(i)$ denotes summation over all $i\in\mathbb{N}$ such that $a\leq i \leq b$, regardless of whether $a,b\in\mathbb{N}$. $1_{\{\cdot\}}$ denotes the indicator function of a set or condition. When proving a probability bound of the form $C\exp\left(-ct^2\right)$, we may take $C$ sufficiently large and assume in the proof that, say, $t\geq1$, because for $t<1$ the resulting probability bound is greater than $1$ and the result holds trivially. After such a bound is proved we may replace $C$ with $2$ using the fact that there exists $c'>0$ such that 
\[
\min\left\{1,C\exp\left(-ct^2\right)\right\}\leq 2\exp\left(-c't^2\right)
\]
Lastly, by making an all-round change of variables we may take $c'$ to be, say, $1$ or $1/2$. The constants in the final probability bound will therefore (without further explanation) not always match what appears to come from the proof.

\section{\label{Lo estimating}Lemmas}
We start with basic estimates for the lower incomplete gamma function suited to our purposes.
\begin{lemma}
\label{incomplete gamma}For all $b,q\in \left[ 0,\infty \right)$,
\begin{eqnarray*}
c^{1+q}\min\left\{1+q,b\right\}^{1+q}&\leq& \int_0^b e^{-\omega}\omega^qd\omega\leq C^{1+q}\min\left\{1+q,b\right\}^{1+q}\\
\frac{ce^bb^{1+q}}{1+q+b}&\leq& \int_0^b e^{\omega}\omega^qd\omega \leq \frac{Ce^bb^{1+q}}{1+q+b}
\end{eqnarray*}
\end{lemma}
\begin{proof}
The first integrand increases on $[0,q]$ and decreases on $[q,\infty)$, so for $b\leq q$, comparing the integral to the area of a large rectangle,
\[
\int_0^b e^{-\omega}\omega^qd\omega \leq be^{-b}b^q\leq b^{1+q}
\]
while for $b\geq1+q$,
\[
\int_0^b e^{-\omega}\omega^qd\omega \leq \Gamma(1+q)\leq C^{1+q}(1+q)^{1+q}\leq C^{1+q}b^{1+q}
\]
If $q\geq1$ this also holds for $q<b<1+q$, since in that case $(1+q)^{1+q}\leq C^{1+q}b^{1+q}$. So all that remains for the upper bound is the case where $0\leq q <1$ and $q<b<1+q$, which implies
\[
\int_0^b e^{-\omega}\omega^qd\omega \leq \int_0^b \omega^qd\omega \leq b^{1+q}
\]
We now consider the lower bound. For $b\leq q$, comparing the integral to the area of a smaller rectangle,
\[
\int_0^b e^{-\omega}\omega^qd\omega \geq \frac{b}{2}e^{-b/2}\left(\frac{b}{2}\right)^q\geq c^{1+q}b^{1+q}
\]
For $b\geq q$ and $q\geq 1$, using what we have just proved,
\[
\int_0^b e^{-\omega}\omega^qd\omega \geq \int_0^q e^{-\omega}\omega^qd\omega \geq c^{1+q}q^{1+q}\geq c^{1+q}(1+q)^{1+q}
\]
For $b\geq q$ and $0\leq q< 1$, we consider two sub-cases: firstly $b\leq 1+q$, in which case
\[
\int_0^b e^{-\omega}\omega^qd\omega \geq c\int_0^b \omega^qd\omega \geq c^{1+q}b^{1+q}
\]
and secondly $b> 1+q$, in which case
\[
\int_0^b e^{-\omega}\omega^qd\omega \geq c \int_0^1 \omega^qd\omega \geq c \geq c^{1+q}(1+q)^{1+q}
\]
The second integral with $e^\omega$ instead of $e^{-\omega}$ can be estimated by writing $e^{\omega}\omega^q=\exp\left( \omega+q\ln \omega\right)$ and using $\ln \omega \leq (\omega-b)/b+\ln b$, valid for all $\omega \in (0,b]$, and $\ln \omega \geq 2(\omega-b)/b+\ln b$, valid for all $\omega \in [b/2,b]$. Here we also use $ce^z/(1+z)\leq(e^z-1)/z\leq Ce^z/(1+z)$ valid for all $z>0$.
\end{proof}

We will use the fact that for any non-increasing function $f:[1,n]\rightarrow \mathbb{R}$,
\[
\frac{1}{2}\left(f(1)+\int_1^nf(x)dx\right)\leq\sum_{i=1}^nf(i)\leq f(1)+\int_1^nf(x)dx
\]

\begin{lemma}
\label{Lo basic sum bound}For all $a,q\in \left[ 0,\infty \right) $ and all $n\geq 2$ the following is
true: If $a\in \left[ 0,1\right] $ then
\[
\sum_{i=1}^{n }i^{-a}\left( \ln \frac{n}{i}\right) ^{q} \leq \frac{C^{1+q}n^{1-a}\left( 1+q\right) ^{1+q}\left( \ln n\right) ^{1+q}}{%
\left( \left( 1-a\right) \ln n+1+q\right) ^{1+q}}
\]%
and if $a\in %
\left[ 1,\infty \right) $ the sum is bounded above by
\[
\frac{C\left( \ln n\right) ^{1+q}}{\left( a-1\right) \ln n+1+q}+\left( \ln
n\right) ^{q}
\]
The corresponding lower bounds hold by replacing $C$ with $c$. When $q=0$ and $i=n$ in the sum, we consider $0^0=1$.
\end{lemma}

\begin{proof}
We focus on the upper bounds; the lower bounds follow the same steps. Integrals are estimated using Lemma \ref{incomplete gamma}, and we use the fact that $\min\{x,y\}$ is the same order of magnitude as $xy/(x+y)$. First, let $a\in \left( 0,\infty \right) $. Peeling off the first
term, comparing the remaining sum to an integral using monotonicity, and
setting $e^{s/a}=n/x$,%
\[
n^{-a}\sum_{i=1}^{n}\left( \frac{%
n}{i}\right) ^{a}\left( \ln \frac{n}{i}\right) ^{q}\leq \left( \ln n\right)
^{q}+\frac{n^{1-a}}{a^{1+q}}\int_0^{a\ln n}\exp \left( \left( 1-\frac{%
1}{a}\right) s\right) s^{q}ds 
\]%
If $a\in \left( 1,\infty \right) $ set $w=\left( 1-1/a\right) s$
to get%
\[
\left( \ln n\right) ^{q}+\frac{n^{1-a}}{\left( a-1\right) ^{1+q}}\int_0^{\left( a-1\right) \ln n}e^{w}w^{q}dw 
\]%
If $a=1$ we get $\left( \ln n\right) ^{q}+\int_{0}^{\ln n}s^{q}ds$ which can be absorbed into either the case $a\in \left( 0,1\right) $ or the case $a\in \left(
1,\infty \right) $. If $a\in \left( 0,1\right) $ then set $w=-\left( 1-1/a\right) s$ to get%
\[
\left( \ln n\right) ^{q}+\frac{n^{1-a}}{\left( 1-a\right) ^{1+q}}\int_0^{\left( 1-a\right) \ln n}e^{-w}w^{q}dw
\]%
If $a=0$, setting $w=\ln \left( n/x\right) $,%
\[
\sum_{i=1}^{n}i^{-a}\left( \ln 
\frac{n}{i}\right) ^{q}\leq \left( \ln n\right) ^{q}+\int_{1}^{n}\left(
\ln \frac{n}{x}\right) ^{q}dx\leq \left( \ln n\right) ^{q}+n\int_0^{\ln n}e^{-w}w^{q}dw 
\]
For $a\in \left[ 0,1\right]$ the factor $\left(\ln n \right)^{q}$ gets absorbed into the remaining term since
\[
\frac{C^{1+q}n^{1-a}\left( 1+q\right) ^{1+q}\left( \ln n\right) ^{1+q}}{%
\left( \left( 1-a\right) \ln n+1+q\right) ^{1+q}}=\frac{C^{1+q}n^{1-a}\left( \ln n\right) ^{1+q}}{%
\left(1+ \frac{\left( 1-a\right) \ln n}{1+q}\right) ^{1+q}}
\]
\end{proof}

The following lemma interpolates between the case $a=1$ and $a\neq 1$.
\begin{lemma}
\label{Lo basic power int}For all $\left(
a,T\right) \in \mathbb{R}\times \left[ 1,\infty \right) $,%
\[
c\frac{1+T^{1-a}}{1+\left\vert 1-a\right\vert \ln T}\ln T\leq
\int_{1}^{T}x^{-a}dx\leq C\frac{1+T^{1-a}}{1+\left\vert 1-a\right\vert \ln T}%
\ln T 
\]
\end{lemma}

\begin{proof}
First assume $a\neq 1$ and $T\neq 1$ and write%
\[
\int_{1}^{T}x^{-a}dx=\frac{\exp \left( \left( 1-a\right) \ln T\right) -1}{%
\left( 1-a\right) \ln T}\ln T 
\]%
Then interpret $s^{-1}\left( \exp \left( s\right) -1\right) $ as the slope
of a secant line and bound it above and below by $C\left( 1+e^{s}\right)
/\left( 1+\left\vert s\right\vert \right) $ in the cases $s\leq -1$, $s\in
\left( -1,1\right) \backslash \left\{ 0\right\} $ and $1\leq s$. Then notice
that the estimate also holds when $a=1$ and/or $T=1$.
\end{proof}

The following lemma is taken from \cite{Fr4} (details can be found in the first arXiv version of \cite{FrL}).
\begin{lemma}
\label{normalorderstats}Let $n\geq 3$, $t\geq 0$, and let $X$ and $Y$ be independent random vectors in $\mathbb{R}^n$, each with the standard normal distribution. With probability at least $1-C\exp\left(-t^2\right)$, the following event occurs: for all $1\leq i \leq (n+1)/2$,
\begin{equation}
X_{\left[ i\right] }\leq C\left( \ln \frac{n}{i}+\frac{t^{2}}{i}\right)
^{1/2} \label{orderstat estimate}
\end{equation}
and with probability at least $0.51$ the following event occurs: for all $1\leq i \leq (n+1)/2$,
\begin{equation}
c\sqrt{\ln \frac{n}{i}}\leq X_{[i]}\leq C\sqrt{\ln \frac{n}{i}}\label{simultaneous medians}
\end{equation}
\end{lemma}

\begin{lemma}
\label{Lo med calculation}Let $0\leq r<\infty$, $0\leq p<\infty$, $n\geq2$, and let $X$ be a random vector in $\mathbb{R%
}^{n}$ with the standard normal distribution. If $r\in \left[ 0,1\right]$ then
\[
\mathbb{M}\sum_{i=1}^{n}i^{-r}X_{\left[ i\right] }^{p}\leq \frac{C^{1+p}(1+p)^{p/2}n^{1-r}\left( \ln n\right) ^{1+p/2}}{\left[1+p+(1-r) \ln n\right] ^{1+p/2}}
\]
and if $r\in \left[ 1,\infty \right)$ then
\[
\mathbb{M}\sum_{i=1}^{n}i^{-r}X_{\left[ i\right] }^{p}\leq \frac{C^{1+p}\left( \ln n\right)^{1+p/2}}{1+(r-1) \ln n}+C^p\left(\ln n\right)^{p/2} 
\]
with the reverse inequalities holding with $C$ replaced by $c$.
\end{lemma}

\begin{proof}
The result follows from Eq. (\ref{simultaneous medians}) of Lemma \ref{normalorderstats}, together with Lemma \ref{Lo basic
sum bound}.
\end{proof}

\begin{lemma}
\label{Lo Lip con}Let $0\leq r<\infty$ and $0<p<\infty$. Then%
\[
\sup \left\{ \left( \sum_{i=1}^{n}i^{-r}\theta _{\left[ i\right]
}^{p}\right) ^{1/p}:\theta \in S^{n-1}\right\} =\left\{ 
\begin{array}{ccc}
\left( \sum_{1}^{n}i^{-2r/(2-p)}\right) ^{(2-p)/2p} & : & p\in \left(0,2\right) \\ 
1 & : & p\in \left[ 2,\infty \right)%
\end{array}%
\right. 
\]%
For $p\in \left[2/3,2\right) $ this can be bounded above by%
\[
1+C\left( \frac{\ln n}{1+\left\vert2-2r-p\right\vert
\ln n}\right) ^{\frac{2-p}{2p}}\left( 1+n^{\frac{2-2r-p}{2p}}\right) 
\]
and below by the same quantity with the leftmost $1$ replaced by $1/2$ and $C$ replaced by $c$. For $r<r_0$ (for any universal constant $r_0>1$) the leftmost `$1+$' (or `$1/2+$') can be deleted. The same comments apply for $p\in(0,2/3)$ by multiplying the entire expression by $C^{1/p}$ for an upper bound and $c^{1/p}$ for a lower bound.
\end{lemma}

\begin{proof}
For $p\in \left(0,2\right) $ an upper bound follows by H\"{o}lder's
inequality for $\ell _{2/(2-p)}^{n}-\ell _{2/p}^{n}$ duality, with equality
when $\theta _{i}=i^{-r/(2-p)}\left( \sum_{j=1}^{n}j^{-2r/(2-p)}\right)
^{-1/2}$. Now 
\[
\left( \sum_{1}^{n}i^{-2r/(2-p)}\right) ^{(2-p)/2p}\leq \left(
1+\int_{1}^{n}x^{-2r/(2-p)}dx\right) ^{(2-p)/2p} 
\]%
which is bounded using Lemma \ref{Lo basic power int} and noting that for $p\in[2/3,2)$,
$0<(2-p)/(2p)\leq 1$ and that for $p\in(0,2)$, $c<(2-p)^{(2-p)}<C$. For $p\in \left[ 2,\infty \right) $, $\left(
\sum_{i=1}^{n}i^{-r}\theta _{\left[ i\right] }^{p}\right) ^{1/p}\leq \left(
\sum_{i=1}^{n}\theta _{\left[ i\right] }^{p}\right) ^{1/p}\leq 1$ with
equality when $\theta =e_{1}$.
\end{proof}

We shall use the fact that for all $b\in [1,n]$, not necessarily an integer,
\[
\sum_{i=1}^{n}i^{-r}x _{\left[ i\right]}^{p}\leq \frac{2n}{b}\sum_{i=1}^{b}i^{-r}x _{\left[ i\right]}^{p}
\]
where the sum on the right is over all $i\in \mathbb{N}$ such that $1\leq i \leq b$.

\section{\label{linords}Statement and proof of the main result}

\begin{theorem}\label{orderorderbound}
Let $X$ be a random vector in $\mathbb{R}^n$, $n\geq3$, with the standard normal distribution, $0\leq r<\infty$, $1\leq p<\infty$ and $t>0$. In each of the following four cases, definitions are given for $R$, $S$ are $\left\vert \cdot \right \vert_{\sharp}$, and in each case,
\[
\mathbb{P}\left\{\left\vert X \right \vert_{\sharp}\leq S\right\}\geq 1-C\exp\left(-t^2/2\right) \hspace{1cm}\textit{and} \hspace{1cm} \left\vert X \right \vert_{\sharp}\leq S \Rightarrow \sum_{i=1}^{n}i^{-2r}X_{[i]}^{2(p-1)}\leq R
\]

\textbf{Case I:} If $p\in \left[3/2,\infty \right)$ then for all $x\in \mathbb{R}^n$ set
\[
 \left\vert x \right \vert_{\sharp}=\left(\sum_{i=1}^{n}i^{-2r}x_{[i]}^{2(p-1)}\right)^{\frac{1}{2(p-1)}}
 \]
 and $R=S^{2(p-1)}=A+Bt^{2(p-1)}$ where
 \[
 A=\frac{C^p p^p n^{1-2r} \left( \ln n \right)^p}{\left[ p+(1-2r) \ln n\right]^p}1_{\{0\leq r\leq 1/2\}}+
 \left(\frac{C^p\left( \ln n \right)^p}{1+(2r-1) \ln n }+C^p\left(\ln n \right)^{p-1}\right)1_{\{1/2< r<\infty\}}
 \]
 and
 \[
 B=C\left[1+\left( \frac{\ln n}{1+\left\vert 2-2r-p \right\vert \ln n} \right)^{2-p}\left( 1+n^{2-2r-p} \right)\right]1_{\{3/2\leq p< 2\}}+C^p1_{\{2\leq p< \infty\}}
 \]
In this case $\left\vert \cdot \right \vert_{\sharp}$ is a norm.

\textbf{Case II:} If $p\in \left[1,3/2 \right)$ then for all $x\in \mathbb{R}^n$,
\[
\sum_{i=1}^{n}i^{-2r}x_{[i]}^{2(p-1)} \leq C\left( \sum_{i=1}^{n/e}i^{-2r}\left( \ln \frac{n}{i}+\frac{t^{2}
}{i}\right) ^{p-1}\right) ^{3-2p}\left\vert x \right \vert_{\sharp}^{2\left(
p-1\right) }
\]
where
\begin{equation}
\left\vert x \right \vert_{\sharp}=\sum_{i=1}^{n/e}\frac{i^{-2r}x_{\left[ i\right] }}{\left( \ln
\frac{n}{i} +\frac{t^{2}}{i}\right) ^{\frac{3-2p}{2}}} \label{nomnom}
\end{equation}
and
\begin{eqnarray*}
R&=&S=C\sum_{i=1}^{n/e}i^{-2r}\left( \ln \frac{n}{i}+\frac{t^{2}
}{i}\right) ^{p-1}\\
&\leq&\frac{Cn^{1-2r}\left(\ln n\right)^p}{\left[1+(1-2r)\ln n \right]^p}1_{\{0\leq r \leq 1/2\}}+
C\left(\frac{\left(\ln n \right)^p}{1+(2r-1)\ln n}+\left(\ln n \right)^{p-1} \right)1_{\{1/2< r<\infty\}}\\
&+&C\left(1+\frac{1+n^{2-2r-p}}{1+\left\vert 2-2r-p\right\vert\ln n}\ln n \right)t^{2(p-1)}
\end{eqnarray*}
Under the added assumption that $p \geq 3/2-2r$, and because the sums have been restricted to $1\leq i \leq n/e$, $\left\vert \cdot \right \vert_{\sharp}$ is a norm.

\textbf{Case III:} If $p<3/2-2r$ (so necessarily $0\leq r <1/4$ and $1\leq p <3/2$), then for all $x\in \mathbb{R}^n$,
\[
\sum_{i=1}^{n}i^{-2r}x_{[i]}^{2(p-1)} \leq Cn^{(1-2r)(3-2p)}\left\vert x \right \vert_{\sharp}^{2(p-1)}
\]
where
\[
\left\vert x \right \vert_{\sharp}=\sum_{i=1}^{n}i^{-2r}x_{[i]}
\]
and
\begin{eqnarray*}
S&=&Cn^{1-2r}+Cn^{\frac{1-4r}{2}}\left( \frac{\ln n}{1+(1-4r) \ln n} \right)^{\frac{1}{2}}t \\
R&=&Cn^{(1-2r)(3-2p)}S^{2(p-1)}\leq Cn^{1-2r}+Cn^{2-2r-p}\left(\frac{\ln n}{1+(1-4r)\ln n}\right)^{p-1}t^{2(p-1)}
\end{eqnarray*}

\textbf{Case IV:} When $r\in \left(1/4,1/2\right]$ and $p=2(1-r)$, in which case $p\in \left[1,3/2 \right)$, the result in Case II can be improved as follows:

\underline{Case IVa:} If $(1-2r)\ln n \geq e$ (which excludes the case $p=1$), then for all $x\in \mathbb{R}^n$,
\[
\sum_{i=1}^{n}i^{-2r}x_{[i]}^{2(p-1)} \leq C\left(\ln n\right)^{3-2p} \left\vert x \right \vert_{\sharp}^{2(p-1)}
\]
where
\[
\left\vert x \right \vert_{\sharp}=\sum_{i=1}^{n/e} \beta_i^{\frac{-(3-2p)}{2(p-1)}}i^{\frac{-r}{p-1}}x_{[i]},
\]
and
\[
\beta_i=\frac{(1-2r)^p \ln n}{n^{1-2r}}i^{-2r}\left( \ln \frac {n}{i}\right)^{p-1}+i^{-1}
\]
The coefficient $\beta_i^{\frac{-(3-2p)}{2(p-1)}}i^{\frac{-r}{p-1}}$ is non-increasing in $i$ for $1\leq i \leq n/e$ (so that $\left\vert \cdot \right \vert_{\sharp}$ is a norm). In this case
\[
S=C^{\frac{1}{p-1}}(1-2r)^{\frac{-p}{2(p-1)}}\left( \ln n \right)^{\frac{-(3-2p)}{2(p-1)}}n^{\frac{1}{2}}+C^{\frac{1}{p-1}}\left( \ln n \right)^{\frac{1}{2}}t
\]
and
\[
R=C\left(\ln n\right)^{3-2p} S^{2(p-1)}\leq C(1-2r)^{-1}n^{1-2r}+C\left( \ln n \right)^{2-p}t^{2(p-1)}
\]

\underline{Case IVb:} If $(1-2r)\ln n < e$, then for all $x\in \mathbb{R}^n$,
\[
\sum_{i=1}^{n}i^{-2r}x_{[i]}^{2(p-1)} \leq C\left( \ln n\right) \left \vert x\right \vert_{\sharp}^{2(p-1)}
\]
where $\left \vert \cdot \right \vert_{\sharp}=\left\vert\cdot\right\vert$ is the standard Euclidean norm, $S=Cn^{\frac{1}{2}}+t$, and
\[
R=C\left( \ln n\right) S^{2(p-1)}\leq C\left( \ln n \right)t^{2(p-1)}
\]
\end{theorem}
\begin{proof}
\textbf{Case I:} $p\in \left[3/2, \infty \right)$. In this case $\left \vert \cdot \right \vert_{2r,2(p-1)}$ is a norm and it follows by classical Gaussian concentration that with probability at least $1-C\exp\left(-t^2/2\right)$,
\[
\sum_{i=1}^{n}i^{-2r}X_{[i]}^{2(p-1)}=\left \vert X \right \vert_{2r,2(p-1)}^{2(p-1)}\leq \left[ \mathbb{M}\left \vert X \right \vert_{2r,2(p-1)}+t \mathrm{Lip}\left \vert \cdot \right \vert_{2r,2(p-1)} \right]^{2(p-1)}
\]
Estimates for $\mathbb{M}\left \vert X \right \vert_{2r,2(p-1)}$ and $\mathrm{Lip}\left \vert \cdot \right \vert_{2r,2(p-1)}$ follow from Lemmas \ref{Lo med calculation} and \ref{Lo Lip con}, and we leave the computation to the reader (with the reminder that we are applying these results with $2r$ and $2(p-1)$ instead of $r$ and $p$). Certain numerical simplifications can be made based on the values of $p$ and $r$ and the existence of the factor $C^p$.

\noindent \textbf{Proof. Case II:} We consider $p\in (1,3/2)$ and reclaim the case $p=1$ by taking a limit. For all $x\in \mathbb{R}^n$, by H\"{o}lder's inequality, $\sum_{1}^{n}i^{-2r}x_{[i]}^{2(p-1)}$ is bounded above by
\begin{eqnarray*}
&&C\sum_{i=1}^{n/e}\frac{i^{-4r\left( p-1\right) }x_{\left[ i\right]
}^{2(p-1)}}{\left( \ln \left( n/i\right) +t^{2}/i\right) ^{\left( p-1\right)
\left( 3-2p\right) }}i^{-2r\left( 3-2p\right) }\left( \ln \left( n/i\right)
+t^{2}/i\right) ^{\left( p-1\right) \left( 3-2p\right) } \\
&\leq&C\left( \sum_{i=1}^{n/e}\frac{i^{-2r}x_{\left[ i\right] }}{\left( \ln
\left( n/i\right) +t^{2}/i\right) ^{\left( 3-2p\right) /2}}\right) ^{2\left(
p-1\right) }\left( \sum_{i=1}^{n/e}i^{-2r}\left( \ln \frac{n}{i}+\frac{t^{2}%
}{i}\right) ^{p-1}\right) ^{3-2p}\\
&= &C\left\vert x \right\vert_{\sharp}^{2(p-1)}\left( \sum_{i=1}^{n/e}i^{-2r}\left( \ln \frac{n}{i}+\frac{t^{2}%
}{i}\right) ^{p-1}\right) ^{3-2p}
\end{eqnarray*}%
Eq. (\ref{orderstat estimate}) from Lemma \ref{normalorderstats} then implies that with high probability
\[
\left\vert X \right\vert_{\sharp}=\sum_{i=1}^{n/e}\frac{i^{-2r}X_{\left[ i\right] }}{\left( \ln\left( n/i\right) +t^{2}/i\right) ^{\left( 3-2p\right) /2}}
\leq C'\sum_{i=1}^{n/e}i^{-2r}\left( \ln \frac{n}{i}+\frac{t^{2}}{i}\right) ^{p-1}=\frac{S}{C''}
\]
Assuming this event occurs, the above calculation involving H\"{o}lder's inequality implies that $\sum_{1}^{n}i^{-2r}X_{[i]}^{2(p-1)}$ is bounded above by $S$. Lastly,
\begin{eqnarray*}
S\leq C\sum_{i=1}^{n/e}i^{-2r}\left( \ln \frac{n}{i}\right)
^{p-1}+Ct^{2(p-1)}\sum_{i=1}^{n/e}i^{-2r-p+1}
\end{eqnarray*}
which is bounded above using Lemmas \ref{Lo basic sum bound} and \ref{Lo basic power int}.

\noindent \textbf{Proof. Case III:} $p<3/2-2r$ (so necessarily $0\leq r <1/4$ and $1\leq p <3/2$). For $p\neq 1$, by H\"{o}lder's inequality,
\[
\sum_{i=1}^{n}i^{-2r}x_{[i]}^{2(p-1)}\leq \left(\sum_{i=1}^{n}i^{-2r}x_{[i]}\right)^{2(p-1)}\left(\sum_{i=1}^{n}i^{-2r}\right)^{3-2p}\leq Cn^{(1-2r)(3-2p)}\left\vert x \right\vert_{\sharp}^{2(p-1)}
\]
and for $p=1$ the same bound is seen to hold. An upper bound on the quantiles of $\left\vert X \right\vert_{\sharp}$ follows from Gaussian concentration (making use of Lemmas \ref{Lo med calculation} and \ref{Lo Lip con}).

\noindent \textbf{Proof. Case IV:} $p=2(1-r)$ and $r\in \left(1/4,1/2\right]$, in which case $p\in \left[1,3/2\right)$. The sub-case $(1-2r)\ln n < e$ is clear enough by H\"{o}lder's inequality for $p\neq 1$,
\begin{eqnarray*}
\sum_{i=1}^{n}i^{-2r}x_{[i]}^{2(p-1)} &\leq& \left( \sum_{i=1}^{n}\left(i^{-2r}\right)^{\frac{1}{2-p}}\right)^{2-p}\left( \sum_{i=1}^{n}\left(x_{[i]}^{2(p-1)}\right)^{\frac{1}{p-1}}\right)^{p-1}
\end{eqnarray*}
and then noting that the exponent $-2r/(2-p)=-1$ and applying classical Gaussian concentration to $\left\vert \cdot \right\vert$. For $p=1$ there is nothing to show. In this sub-case,
\[
\frac{\left(\ln n \right)^{2-p}}{\ln n}=\exp \left( -(1-2r)\ln\ln n \right)\in \left[e^{-1},1\right] \hspace{1.5cm} n^{1-2r}=\exp\left((1-2r)\ln n\right)\in \left[1,e^e\right)
\]
which is how we simplify the exponents of $\ln n$.

The rest of the proof deals with the other sub-case $(1-2r)\ln n \geq e$. Throughout, we make use of the relation $p=2(1-r)$ which is not always explicitly re-stated, and the reader should make a mental note of this. The case $p=1$ is automatically excluded from this sub-case. By taking the constant $C$ in the probability bound to be at least $\sqrt{e}$ we may assume that $t\geq1$. For any sequence $\left(\alpha_i\right)_1^{\left \lfloor n/e\right \rfloor}$ with $\alpha_i>0$, by H\"{o}lder's inequality,
\begin{eqnarray*}
\sum_{i=1}^{n}i^{-2r}x_{[i]}^{2(p-1)} &\leq&C\left(\sum_{i=1}^{n/e}\alpha_i^{\frac{1}{3-2p}}\right)^{3-2p}\left( \sum_{i=1}^{n/e}i^{\frac{-r}{p-1}}\alpha_i^{\frac{-1}{2(p-1)}}x_{[i]} \right)^{2(p-1)}\\
&=& C\left(\sum_{i=1}^{n/e}\beta_{i}\right)^{3-2p}\left( \sum_{i=1}^{n/e}i^{\frac{-r}{p-1}}\beta_{i}^{\frac{-(3-2p)}{2(p-1)}}x_{[i]} \right)^{2(p-1)}
\end{eqnarray*}
where $\beta_i=\alpha_i^{\frac{1}{3-2p}}$. This vector $\beta\in\mathbb{R}^{\left \lfloor n/e\right \rfloor}$ is considered a variable for now, and its value will later be fixed to match the value quoted in the statement of the result. Summing only up to $n/e$ will ensure that $i^{\frac{2r}{3-2p}}\beta_i$ is non-decreasing in $i$, which then implies that
\begin{equation}
\left\vert x\right\vert_{\sharp}=\sum_{i=1}^{n/e}i^{\frac{-r}{p-1}}\beta_{i}^{\frac{-(3-2p)}{2(p-1)}}x_{[i]} \label{normee}
\end{equation}
is a norm. By classical Gaussian concentration applied to $\left\vert \cdot \right\vert_{\sharp}$, with probability at least $1-C\exp\left(-t^2/2\right)$,
\begin{equation}
\sum_{i=1}^{n}i^{-2r}X_{[i]}^{2(p-1)}\leq \left(\sum_{i=1}^{n/e}\beta_{i}\right)^{3-2p} \left[\mathbb{M} \left\vert X \right\vert_{\sharp} +t \mathrm{Lip}\left( \left\vert \cdot \right\vert_{\sharp}\right) \right]^{2(p-1)} \label{essential bound}
\end{equation}
The median can be estimated using (\ref{orderstat estimate}) and the Lipschitz constant computed as the Euclidean norm of the gradient, which gives
\begin{eqnarray*}
\mathbb{M} \left\vert X \right\vert_{\sharp} &\leq& C\sum_{i=1}^{n/e} i^{\frac{-r}{p-1}} \beta_i^{\frac{-(3-2p)}{2(p-1)}}\left(\ln \frac{n}{i}\right)^{1/2}\\
\mathrm{Lip}\left( \left\vert \cdot \right\vert_{\sharp}\right) &=&\left(\sum_{i=1}^{n/e} i^{\frac{-2r}{p-1}} \beta_i^{\frac{-(3-2p)}{p-1}} \right)^{1/2}
\end{eqnarray*}
We temporarily assume that $\sum\beta_{i}=1$, which we may do by homogeneity, although this condition will later be relaxed. We wish to minimize the function
\[
\psi\left( \beta \right)=\sum_{i=1}^{n/e} i^{\frac{-r}{p-1}} \beta_i^{\frac{-(3-2p)}{2(p-1)}}\left(\ln \frac{n}{i}\right)^{1/2}+t\left(\sum_{i=1}^{n/e} i^{\frac{-2r}{p-1}} \beta_i^{\frac{-(3-2p)}{p-1}} \right)^{1/2}
\]
over the collection of all $\beta\in \mathbb{R}^{\left\lfloor n/e \right\rfloor}$ such that $i^{\frac{2r}{3-2p}}\beta_i$ is positive and non-decreasing in $i$ and such that $\sum\beta_{i}=1$. The method of Lagrange multipliers leads us to solve the equations
\[
\frac{\partial \psi\left(\beta\right)}{\partial \beta_i}=-\lambda
\]
which can be written as
\[
B_1i^{\frac{-r}{p-1}}\left( \ln \frac{n}{i}\right)^{1/2}\beta_i^{\frac{-1}{2(p-1)}}+B_2i^{\frac{-2r}{p-1}}\beta_i^{\frac{-(2-p)}{p-1}}=1
\]
where $B_1$ and $B_2$ are positive values that do not depend on $i$. This implies that
\[
1/2 \leq \max \left\{B_1i^{\frac{-r}{p-1}}\left( \ln \frac{n}{i}\right)^{1/2}\beta_i^{\frac{-1}{2(p-1)}},B_2i^{\frac{-2r}{p-1}}\beta_i^{\frac{-(2-p)}{p-1}}\right\}\leq 1
\]
and therefore
\[
\beta_i\leq \max \left\{2^{2(p-1)}B_1^{2(p-1)}i^{-2r}\left( \ln \frac{n}{i}\right)^{p-1},2^{\frac{p-1}{2-p}}B_2^{\frac{p-1}{2-p}}i^{\frac{-2r}{2-p}} \right\}
\]
with the reverse inequality holding when $2^{2(p-1)}$ and $2^{\frac{p-1}{2-p}}$ are deleted. At this point, and by homogeneity, we remove the condition $\sum \beta_i=1$ and are led to the definition
\[
\beta_i=Ai^{-2r}\left( \ln \frac{n}{i}\right)^{p-1}+i^{-1}
\]
for some $A>0$. $B_1$, $B_2$ and the powers of $2$ disappear since they do not depend on $i$ and we have re-scaled $\beta$, and we have used the equation $p=2(1-r)$ to simplify the exponent $-2r/(2-p)$. We now minimize over $A$. It follows from Lemma \ref{Lo basic sum bound} that
\begin{eqnarray*}
\sum_{i=1}^{n/e}\beta_i &\leq& CA(1-2r)^{-p}n^{1-2r}+C\ln n
\end{eqnarray*}
With an eye on (\ref{essential bound}), it is clear that the bounds for $\mathbb{M} \left\vert X \right\vert_{\sharp}$ and $\mathrm{Lip}\left( \left\vert \cdot \right\vert_{\sharp}\right)$ are decreasing in $A$. It therefore does not help to let $A$ slip below the point where
\[
CA(1-2r)^{-p}n^{1-2r}=C\ln n
\]
because as $A$ continues to decrease beyond this point $\sum \beta_i$ stays the same order of magnitude while $\mathbb{M} \left\vert X \right\vert_{\sharp}+t\mathrm{Lip}\left( \left\vert \cdot \right\vert_{\sharp}\right)$ increases. We may therefore assume that
\[
A\geq \frac{c(1-2r)^{p}\ln n}{n^{1-2r}} \hspace{3cm} \sum_{i=1}^{n/e}\beta_i \leq CA(1-2r)^{-p}n^{1-2r}
\]
If we look back at (\ref{essential bound}) with our new bound for $\sum \beta_i$ and our definition of $\beta_i$, and we take $A$ out of the expression for $\sum \beta_i$ and move it into the powers of $\beta_i$ with corresponding exponents $\frac{-(3-2p)}{2(p-1)}$ and $\frac{-(3-2p)}{p-1}$ in the expression $\mathbb{M} \left\Vert X \right\Vert +t \mathrm{Lip}\left( \left\Vert \cdot \right\Vert\right)$, we see that these powers of $\beta_i$ become
\[
\left( i^{-2r}\left( \ln \frac{n}{i}\right)^{p-1}+A^{-1}i^{-1} \right)^{\frac{-(3-2p)}{2(p-1)}} \hspace{2cm} \left( i^{-2r}\left( \ln \frac{n}{i}\right)^{p-1}+A^{-1}i^{-1} \right)^{\frac{-(3-2p)}{p-1}}
\]
So, in our current range for $A$, the expression to be minimized (or at least the bound that we have for it) is increasing. This leads us to take
\begin{equation}
A= \frac{(1-2r)^{p}\ln n}{n^{1-2r}} \label{Adef}
\end{equation}
Recall that for $\left\vert \cdot \right\vert_{\sharp}$ to be a norm, see (\ref{normee}), it is sufficient for $i^{2r/(3-2p)}\beta_i$ to be non-decreasing in $i$, equivalently for $i^{-r/(p-1)}\beta_{i}^{-(3-2p)/(2p-2)}$ to be non-increasing. Writing 
\[
\omega_i=i^{\frac{-r}{p-1}}\beta_{i}^{\frac{-(3-2p)}{2(p-1)}}=\left[An^{\frac{4r(p-1)}{3-2p}}\left(\frac{n}{i} \left( \ln \frac{n}{i} \right)^{\frac{-(3-2p)}{4r}}\right)^{\frac{-4r(p-1)}{3-2p}}+i^{\frac{p-1}{3-2p}}\right]^{\frac{-(3-2p)}{2(p-1)}}
\]
and noting that $z\left( \ln z\right)^{-(3-2p)/(4r)}$ is increasing for $z\geq \exp \left((3-2p)/(4r)\right)$, we see that $\omega_i$ is decreasing. We now bound $\mathbb{M} \left\vert X \right\vert_{\sharp}$ and $\mathrm{Lip}\left( \left\vert \cdot \right\vert_{\sharp}\right)$. From the definition of $\beta_i$,
\begin{equation}
\beta_i^{-1} \leq \min\left\{A^{-1}i^{2r}\left(\ln \frac{n}{i}  \right)^{-(p-1)}, i\right\} \label{beta inv}
\end{equation}
which leads us to solve,
\begin{eqnarray*}
A^{-1}i^{2r}\left(\ln \frac{n}{i}  \right)^{-(p-1)}=i
\end{eqnarray*}
Keeping in mind that $1-2r=p-1$, the above equation holds precisely when
\begin{equation}
\frac{n}{i} \left(\ln \frac{n}{i}\right)^{-1}=A^{\frac{1}{p-1}}n \label{eoi}
\end{equation}
The function $z\mapsto z/\ln z$ is increasing on $[e,\infty)$ and we will show that provided $n>n_0$ (for a universal constant $n_0>1$),
\begin{equation}
\frac{1}{2}e^{2}\leq A^{\frac{1}{p-1}}n \leq \frac{2e}{3\ln n}n^{\frac{3}{2e}} \label{rangee}
\end{equation}
so that (\ref{eoi}) has exactly one solution for $i\in [n^{1-\frac{3}{2e}},n/e^2]$, denoted $A_0$ (not necessarily an integer) which satisfies
\begin{eqnarray}
n^{1-\frac{3}{2e}}&\leq& A_0\leq e^{-2}n\\
A_0 \ln \frac{n}{A_0}&=&A^{\frac{-1}{p-1}} \label{basic AA0}\\
\ln \left(A^{1/(p-1)}n \right)&=& \ln \frac{n}{A_0} - \ln \ln \frac{n}{A_0}\\
\ln \left(A^{1/(p-1)}n \right) &\leq& \ln \left( \frac{n}{A_0}\right) \leq \left( 1-\frac{1}{e}\right)^{-1} \ln \left(A^{1/(p-1)}n \right) \label{logsall}
\end{eqnarray}
The assumption $n>n_0$ does not limit our generality since the result is directly seen to hold when $n\leq n_0$ in which case many of the coefficients involved are bounded by constants. From (\ref{logsall}), (\ref{Adef}) and the defining inequality of the current sub-case, it follows that
\[
\frac{1}{2}\ln \left((1-2r)\ln n \right) \leq (1-2r)\ln \left( \frac{n}{A_0}\right) \leq \left( 1-\frac{1}{e}\right)^{-1} \ln \left((1-2r)\ln n \right)
\]
For the left inequality we used the fact that $p-1\in (0,1/2)$ and $z\ln z \geq -1/e$ for $z\in (0,1/2)$. The right inequality is more straightforward. We now verify (\ref{rangee}). Recalling (\ref{Adef}) and the fact that $1-2r=p-1$, which we are using constantly, the lower bound in (\ref{rangee}) holds provided
\[
\ln n \geq \frac{e^{(2-\ln 2)(p-1)}}{(p-1)^p}
\]
From the definition of the current sub-case, $\ln n \geq e/(p-1)$, so a sufficient condition for the above inequality to hold is
\[
2-\ln2 \leq \frac{1}{p-1}+\ln (p-1)
\]
which is true by considering $1/z+\ln z$ for $z\in \left( 0,1/2 \right)$. The upper bound in (\ref{rangee}) holds provided
\[
(p-1)\ln n \leq \exp\left( \frac{p-1}{p} \ln \left(\frac{2e}{3}n^{\frac{3}{2e}}\right)\right)
\]
which holds by applying $e^z\geq ez$. This completes the task of verifying (\ref{rangee}). From (\ref{beta inv}) it follows that
\begin{equation}
\beta_i^{-1}\leq
\left\{ 
\begin{array}{ccc}
Ci & : & i\leq A_0 \\ 
CA^{-1}i^{2r}\left(\ln \frac{n}{i}\right)^{-(p-1)} & : & i> A_0%
\end{array}%
\right. \label{piecewise beta}
\end{equation}
In an integral where the integrand grows or decays at a controlled rate, one can change an upper bound of $A_0+1$ to $A_0$ at the expense of a constant. Using
\[
 \int_{a}^{b}e^{-\omega}\omega^{1/2}d\omega \leq \frac{C(b-a)e^{-a}a^{1/2}}{1+b-a}
\]
valid as long as $1\leq a \leq b$, and
\[
\int_{0}^{b}e^{-\omega}\omega^{p-1}d\omega \leq C
\]
which gives the correct order of magnitude for (say) $b\geq 1/2$, $\mathbb{M}\left\vert X \right\vert_{\sharp}$ is bounded above by
\begin{eqnarray*}
&&C^{\frac{1}{p-1}}\sum_{i=1}^{A_0}i^{-\frac{1}{2}}\left( \ln \frac{n}{i} \right)^{\frac{1}{2}}+C^{\frac{1}{p-1}}A^{\frac{-(3-2p)}{2(p-1)}}\sum_{i=A_0}^{n/e}i^{-2r}\left( \ln \frac{n}{i} \right)^{p-1}\\
&\leq& C^{\frac{1}{p-1}}n^{-\frac{1}{2}}\int_1^{A_0}\left( \frac{n}{x}\right)^{1/2}\left( \ln \frac{n}{x}\right)^{1/2}dx+C^{\frac{1}{p-1}}A^{\frac{-(3-2p)}{2(p-1)}}n^{-2r}\int_{A_0}^{n/e}\left( \frac{n}{x}\right)^{2r}\left( \ln \frac{n}{x}\right)^{p-1}dx\\
&\leq& C^{\frac{1}{p-1}}n^{\frac{1}{2}}\int_{\frac{1}{2}\ln \frac{n}{A_0}}^{\frac{1}{2}\ln n}e^{-\omega}\omega^{1/2}d\omega +C^{\frac{1}{p-1}}A^{\frac{-(3-2p)}{2(p-1)}}(1-2r)^{-p}n^{1-2r}\int_{1-2r}^{(1-2r)\ln \frac{n}{A_0}}e^{-\omega}\omega^{p-1}d\omega\\
&\leq& C^{\frac{1}{p-1}}A_0^{\frac{1}{2}}\left(\ln \frac{n}{A_0}\right)^{\frac{1}{2}}+C^{\frac{1}{p-1}}A^{\frac{-(3-2p)}{2(p-1)}}(1-2r)^{-p}n^{1-2r}\\
&\leq& C^{\frac{1}{p-1}}A^{\frac{-(3-2p)}{2(p-1)}}(1-2r)^{-p}n^{1-2r}
\end{eqnarray*}
We claim that
\begin{eqnarray*}
\int_0^be^{\omega}\omega^{-(3-2p)}d\omega \leq
\left\{ 
\begin{array}{ccc}
C(p-1)^{-1}b^{2(p-1)} & : & 0\leq b\leq 1 \\ 
C(p-1)^{-1}+Ce^bb^{-(3-2p)} & : & b \geq 1%
\end{array}%
\right.
\end{eqnarray*}
For $0\leq b \leq 1$ this is clear. For $b\geq3$ this follows because on $[2,\infty)$ the local exponential growth rate of the integrand is
\[
\frac{d}{d\omega}\left[\omega-(3-2p)\ln \omega\right]=1-\frac{3-2p}{\omega}\in\left[0.5,1\right]
\]
and for $1<b<3$ the bound follows by monotonicity in $b$. Using the claim just proved, $\textrm{Lip}\left(\left\vert \cdot \right\vert_{\sharp} \right)$ is bounded above by
\begin{eqnarray*}
&&\left[C^{\frac{1}{p-1}}\ln A_0+C^{\frac{1}{p-1}}A^{\frac{-(3-2p)}{p-1}}n^{-4r}\int_{A_0}^{n/e}\left( \frac{n}{x}\right)^{4r}\left( \ln \frac{n}{x}\right)^{-(3-2p)}dx \right]^{\frac{1}{2}}\\
&\leq&C^{\frac{1}{p-1}}\left(\ln A_0 \right)^{\frac{1}{2}}+C^{\frac{1}{p-1}}A^{\frac{-(3-2p)}{2(p-1)}}(4r-1)^{1-p}n^{\frac{1-4r}{2}}\left(\int_{4r-1}^{(4r-1)\ln \frac{n}{A_0}}e^{\omega}\omega^{-(3-2p)}d\omega \right)^{\frac{1}{2}}
\end{eqnarray*}
If $(4r-1)\ln \left(n/A_0\right)<1$ then this is bounded by
\[
C^{\frac{1}{p-1}}\left(\ln n \right)^{\frac{1}{2}}+C^{\frac{1}{p-1}}A^{\frac{-(3-2p)}{2(p-1)}}n^{\frac{1-4r}{2}}\left(\ln \frac{n}{A_0}\right)^{p-1}\leq C^{\frac{1}{p-1}}\left(\ln n \right)^{\frac{1}{2}}
\]
To see why the last inequality is true, note that the inequality
\[
\left(\ln n \right)^{\frac{1}{2}}\geq A^{\frac{-(3-2p)}{2(p-1)}}n^{\frac{1-4r}{2}}\left(\ln \frac{n}{A_0}\right)^{p-1}
\]
reduces to
\[
A_0^{\frac{3-2p}{2}}\left(\ln \frac{n}{A_0}\right)^{\frac{1}{2}}\leq n^{\frac{4r-1}{2}}\left(\ln n\right)^{\frac{1}{2}}
\]
which in turn follows since $1\leq A_0\leq n$ and $3-2p=4r-1$. If $(4r-1)\ln \left(n/A_0\right)\geq1$ then using $4r-1=3-2p$ and $A_0\ln \left(n/A_0\right)=A^{-1/(p-1)}$, we get the same bound, i.e.
\[
\textrm{Lip}\left(\left\vert \cdot \right\vert_{\sharp} \right)\leq C^{\frac{1}{p-1}}\left(\ln n \right)^{\frac{1}{2}}
\]
Going all the way back to (\ref{essential bound}), regardless of whether $(4r-1)\ln \left(n/A_0\right)$ lies in $[0,1)$ or $[1,\infty)$,
\[
\sum_{i=1}^ni^{-2r}X_{[i]}^{2(p-1)}\leq C (1-2r)^{-p}n^{1-2r}+Ct^{2(p-1)}\left( \ln n \right)^{2-p}
\]
Then note that
\[
\frac{(1-2r)^{-p}}{(1-2r)^{-1}}=\exp \left( -(p-1) \ln (p-1) \right) \in (c,1)
\]
\end{proof}

\section{Simplified bounds and sharpness}

The bounds in Theorem \ref{orderorderbound} simplify if we ignore the norm $\left\vert \cdot \right\vert _\sharp$ and focus on the distributional estimates for $\sum_{i=1}^{n}i^{-2r}X_{\left[ i\right]
}^{2(p-1)}$, allowing constants that depend on $r$ and $p$, and limiting our attention to $n>n_{r,p}$ and the extreme cases $t=C$ and $t>t_{n,r,p}$. We also focus on the case $0<p<3/2$ which is the most interesting case, and we limit our attention to $r\leq 2$ so that the estimates simplify further, although the constant $2$ is arbitrary. With these simplifications we see that the bounds in Theorem \ref{orderorderbound} are sharp in a sense that is made clear in Remark \ref{shar simpy} below. Note that this remark is about Theorem \ref{orderorderbound} as much as it is about the median and probability in question (bounds on the median were already given in Lemma \ref{Lo med calculation}).

The cases corresponding to the diagram below are labelled starting from ii and include $p=1$ because these cases (and others) are considered in \cite{FrL}.

\begin{center}
\begin{tikzpicture}
\fill[black!05!white] (0,4.5) -- (0.75,3) -- (1.5,3) -- (1.5,4.5) -- cycle;
\fill[black!08!white] (1.5,3) -- (3,3) -- (3,4.5) -- (1.5,4.5) -- cycle;
\fill[black!05!white] (3,3) -- (6,3) -- (6,4.5) -- (3,4.5) -- cycle;
\fill[black!10!white] (0,3) -- (0.75,3) -- (0,4.5) -- cycle;
\draw[thick,->] (0,3) -- (6.5,3) node[anchor=north west] {$r$};
\draw[thick,->] (0,3) -- (0,5) node[anchor=south east] {$p$};
\draw (0,3.035) -- (0,2.965) node[anchor=north] {$0$};
\draw (0.75,3.035) -- (0.75,2.965) node[anchor=north] {$\frac{1}{4}$};
\draw (1.5 cm,3.035) -- (1.5 cm,2.965) node[anchor=north] {$\frac{1}{2}$};
\draw (3 cm,3.035) -- (3 cm,2.965) node[anchor=north] {$1$};
\draw (6 cm,3.035) -- (6 cm,2.965) node[anchor=north] {$2$};
\draw (1pt,3 cm) -- (-1pt,3 cm) node[anchor=east] {$1$};
\draw (1pt,4.5 cm) -- (-1pt,4.5 cm) node[anchor=east] {$\frac{3}{2}$};

\draw[black!17.5!white,dashed] (0.75,3) -- (0.75,5);
\draw[dashed] (0,3) -- (6,3);
\draw[dashed] (0,4.5) -- (6,4.5);
\draw[dashed] (3,3) -- (3,5);
\draw[dashed] (6,3) -- (6,5);
\draw[dashed] (1.5,3) -- (1.5,5);
\draw[dashed] (0,4.5) -- (0.75,3);
\draw[thick] (0.75,4.5) -- (1.5,3);
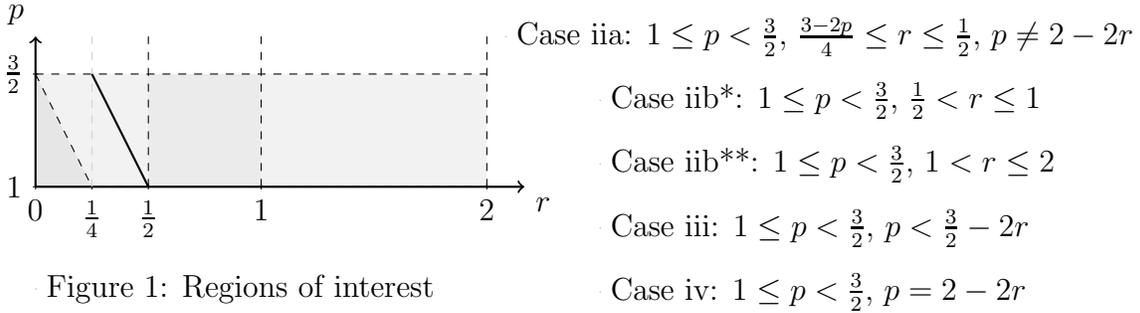
\filldraw[black](0,1.64) circle (0pt) node[anchor=west]{Figure 1: Regions of interest};
\filldraw[black](6.25,5) circle (0pt) node[anchor=west]{Case iia: $1\leq p<\frac{3}{2}$, $\frac{3-2p}{4}\leq r \leq \frac{1}{2}$, $p\neq 2-2r$};
\filldraw[black](7.5,4.15) circle (0pt) node[anchor=west]{Case iib*: $1\leq p<\frac{3}{2}$, $\frac{1}{2}< r \leq 1$};
\filldraw[black](7.5,3.3) circle (0pt) node[anchor=west]{Case iib**: $1\leq p<\frac{3}{2}$, $1< r \leq 2$};
\filldraw[black](7.5,2.45) circle (0pt) node[anchor=west]{Case iii: $1\leq p<\frac{3}{2}$, $p< \frac{3}{2}-2r$};
\filldraw[black](7.5,1.6) circle (0pt) node[anchor=west]{Case iv: $1\leq p<\frac{3}{2}$, $p=2-2r$};
\end{tikzpicture}
\end{center}

\begin{remark}\label{shar simpy}
Setting $t=C$ in Theorem \ref{orderorderbound} and taking $1<p <3/2$, one gets the bound
\[
\mathbb{M}\sum_{i=1}^{n}i^{-2r}X_{\left[ i\right]}^{2(p-1)}\leq \left\{
\begin{array}{lll}
Cn^{1-2r}\left(\ln n\right)^p\left[1+(1-2r)\ln n\right]^{-p} &:& 0\leq r\leq 1/2\\
C\left(\ln n\right)^p\left[1+(2r-1)\ln n\right]^{-1} &:& 1/2\leq r\leq 2
\end{array}
\right.
\]
and the reverse inequality holds by replacing $C$ with $c$. For $1<p<3/2$, $0\leq r\leq 2$ and $t>t_{n,r,p}$ the bounds in Theorem \ref{orderorderbound} can be written as
\begin{eqnarray*}
&&\mathbb{P}\left\{\sum_{i=1}^{n}i^{-2r}X_{\left[ i\right]}^{2(p-1)}>t \right\}\\
&&\leq \left\{
\begin{array}{lll}
2\exp\left(-c_{r,p}\left[1+n^{(2-2r-p)/(p-1)}\right]^{-1}t^{1/(p-1)}\right)&:& p\neq 2-2r ,n>n_{r,p}\\
2\exp\left(-\left(\ln n\right)^{-(2-p)/(p-1)}(ct)^{1/(p-1)}\right) &:& p=2-2r
\end{array}
\right.
\end{eqnarray*}
and the reverse inequality holds by replacing $c$ and $c_{r,p}$ with $C$ and $C_{r,p}$.
\end{remark}
\begin{proof}[Explanation for Remark \ref{shar simpy}]
For the upper bound of the median in Case iia, we use the fact that $Cn^{1-2r}\left(\ln n\right)^p\left[1+(1-2r)\ln n\right]^{-p}$ is at least
\begin{eqnarray*}
\frac{C\max\{1,n^{1-2r}\}\ln n}{1+\ln \max\{1,n^{1-2r}\}}\geq\frac{C\max\{1,n^{2-2r-p}\}\ln n}{1+\ln \max\{1,n^{2-2r-p}\}}\geq\frac{C\left(1+n^{2-2r-p}\right)\ln n}{1+\left\vert 2-2r-p\right\vert \ln n}
\end{eqnarray*}
so the term involving $t^{2(p-1)}$ does not come into play. Similarly for Case iii,
\[
Cn^{2-2r-p}\left(\frac{\ln n}{1+(1-4r)\ln n}\right)^{p-1}=Cn^{1-2r}\left(\frac{\ln n}{n\left[1+(1-4r)\ln n\right]}\right)^{p-1}\leq Cn^{1-2r}
\]
and for Case iv we consider Case IVa of Theorem \ref{orderorderbound} noting that
\[
C(1-2r)^{-1}n^{1-2r}=\frac{C\exp\left((1-2r)\ln n\right)}{(1-2r)\ln n}\ln n\geq C(\ln n)^{2r}=C(\ln n)^{2-p}
\]
and in Case IVb of Theorem \ref{orderorderbound} we note that (as explained in the proof of that theorem), $\ln n$ and $\left(\ln n\right)^p$ have the same order of magnitude. In all cases, the inequality for the median can be reversed by Lemma \ref{Lo med calculation}. We now explain the lower estimates for the probability. For $t>t_{n,r,p}$
\begin{eqnarray}
\mathbb{P}\left\{b_{n, 2r,2(p-1)}\left\vert X \right\vert\geq (1.01)t\right\}&\leq&\mathbb{P}\left\{\left( \sum_{i=1}^{n}i^{-2r}X_{\left[ i\right]
}^{2(p-1)}\right) ^{1/[2(p-1)]}\geq t\right\}\label{loer}\\
&\leq&\mathbb{P}\left\{b_{n, 2r,2(p-1)}\left\vert X \right\vert\geq t\right\}\label{uppup}
\end{eqnarray}
where $b_{n,r,p}$ is the supremum of the quasi-norm $\left\vert \cdot \right\vert_{r,p}$ over $S^{n-1}$, as estimated in Lemma \ref{Lo Lip con} (here this is not a Lipschitz constant). The upper bound (\ref{uppup}), which holds for all $t>0$, is the bound on the Gaussian measure of a Lorentz ball using the inscribed Euclidean ball. The lower bound (\ref{loer}) follows from \cite[Proposition 9]{Fr} (one of the examples mentioned in the introduction). This estimate can then be written as
\begin{equation}
\mathbb{P}\left\{\left( \sum_{i=1}^{n}i^{-2r}X_{\left[ i\right]
}^{2(p-1)}\right) ^{1/[2(p-1)]}\geq t\right\}\leq 2\exp\left(-c\left(\frac{t}{b_{n, 2r,2(p-1)}}\right)^2\right)\label{gauss measure shape}
\end{equation}
with the reverse inequality by replacing $c$ with $C$.
\end{proof}

\end{document}